\documentclass{article}

\usepackage[english]{babel}

\usepackage[letterpaper, margin=1in]{geometry}

\usepackage{amsmath}
\usepackage{graphicx}
\usepackage{amsfonts}
\usepackage[colorlinks=true, allcolors=blue]{hyperref}

\usepackage{subfig}
\usepackage{amssymb}
\usepackage{amsthm}

\usepackage{cancel}
\usepackage{comment}

\usepackage{algorithm} 
\usepackage{algpseudocode} 

\usepackage{authblk}

\usepackage{enumitem}

\def\clo#1{\overline{#1}}

\def\text#1{\mbox{#1}}

\newtheorem{theorem}{Theorem}%[section]
\newtheorem{lemma}{Lemma}%[section]

\graphicspath {{Figures/}}

\providecommand{\keywords}[1]{
\textbf{\textit{Keywords---}} #1}

\title{\textbf{Simultaneous determination of wave speed, diffusivity and nonlinearity in the Westervelt equation using complex time-periodic solutions}\footnote{The work of B. Palacios was partially supported by Agencia Nacional de Investigaci\'on y Desarrollo (ANID), Grant FONDECYT Iniciaci\'on N$^\circ$11220772. The work of S. Acosta was partially supported by NIH award 1R15EB035359-01A1. S. Acosta would like to thank the support and research-oriented environment provided by Texas Children's Hospital.}}

\author[1]{Sebastian Acosta}
\author[2]{Benjamin Palacios}

\affil[1]{\small Department of Pediatrics, Baylor College of Medicine and Texas Children's Hospital, Houston, TX, USA}

\affil[2]{\small Department of Mathematics, Pontificia Universidad Cat\'olica de Chile, Santiago, Chile}

\begin{document}
\maketitle

\begin{abstract}
We consider an inverse problem governed by the Westervelt equation with linear diffusivity and quadratic-type nonlinearity. The objective of this problem is to recover all the coefficients of this nonlinear partial differential equation. We show that, by constructing complex-valued time-periodic solutions excited from the boundary time-harmonically at a sufficiently high frequency, knowledge of the first- and second-harmonic Cauchy data at the boundary is sufficient to simultaneously determine the wave speed, diffusivity and nonlinearity in the interior of the domain of interest.
\end{abstract}

\keywords{Inverse problems, nonlinear ultrasound imaging, harmonic imaging, Westervelt equation}

%%%%%%%%%%%%%%%%%%%%%%%%%%%%%%%%%%%
%%%%%%%%%  NEW SECTION  %%%%%%%%%%%
%%%%%%%%%%%%%%%%%%%%%%%%%%%%%%%%%%%

\section{Introduction}
\label{Section.Intro}

Nonlinear ultrasound has gained significant importance in medical applications. This includes enhancing the visualization of blood perfusion in organs and tumors through imaging techniques \cite{Anvari2015, Cosgrove2010,Demi2014a,Szabo2004a,Thomas1998, Webb2003}, ablating abnormal tissues using high-intensity focused ultrasound \cite{Chapelon2016, Knuttel2016,terHaar2016}, and effectively delivering drugs or genetic material with the aid of microscopic agents that respond to ultrasound excitation \cite{Bettinger2016,Bouakaz2016b,Castle2016}. The portability of ultrasound-based technologies further increases their appeal, particularly in surgical settings, where the clinical intervention and blood perfusion could be monitored in the operating room and in real-time.

The nonlinearity in the acoustic field can be induced by the natural process governing ultrasound waves with sufficiently large amplitude (tissue harmonics) or by artificial microbubble agents (contrast harmonics) \cite{Burns2000,Anvari2015,Cosgrove2010,Demi2014a,Eyding2020,Szabo2004a}. Either way, the nonlinearity results in the generation of vibrations at frequencies distinct (higher harmonics) from the incoming ultrasound frequency (fundamental harmonic). By detecting and analyzing these unique frequency components, images can be created that specifically highlight the source of the nonlinearity. In fact, the nonlinearity is produced only at locations where high amplitude of oscillation occurs, such as the axis or focal point of acoustic beams. Therefore, harmonic imaging helps reduce several types of ultrasound artifacts.  It filters out unwanted effects at the original (fundamental) frequency, including those caused by clutter, grating lobes, and side lobes of the beams. Because these artifacts do not create higher harmonic frequencies, they are effectively suppressed.  Additionally, the technique minimizes the impact of multiple scattering and echoes, which generates weak signals that also lack harmonic content due to their low amplitude. Hence, this approach effectively suppresses some of the interference caused by the heterogeneous acoustic properties of biological tissues, improving the clarity and accuracy of the imaging.

Regarding mathematical analysis, the imaging problem governed by the Westervelt equation has received recent attention in the groundbreaking work of Kaltenbacher and Rundell \cite{Kaltenbacher2022,Kaltenbacher2024b,Kaltenbacher2021b,Kaltenbacher2021a,Kaltenbacher2022a,Kaltenbacher2023}. They have shown the identifiability of the nonlinearity and/or wave speed (from finitely many boundary measurements) using linearized models of the inverse problem which under appropriate conditions allows for the application of the inverse function theorem to arrive at local uniqueness for the nonlinear inverse problem. In \cite{Acosta2022}, Acosta, Uhlmann and Zhai showed the determination of the coefficient of nonlinearity from knowledge of the whole Dirichlet-to-Neumann (DtN) map for a lossless Westervelt equation (with quadratic nonlinearity) using Gaussian beam solutions. Uhlmann and Zhang considered a more general type of nonlinearity in the form of a power series expansion \cite{Uhlmann2023}. Eptaminitakis and Stefanov showed that the coefficient of nonlinearity can be reconstructed by reducing the problem to the X-ray transform in the weakly nonlinear regime using high-frequency packets \cite{Eptaminitakis2024}. Recently, Li and Zhang addressed the problem of recovering time-dependent potential and nonlinear coefficient \cite{Li2024b} and the metric of a compact Riemannian manifold \cite{Li2024a} from knowledge of the DtN map associated with Westervelt equation. We refer the reader to \cite{Kaltenbacher2025review2} for a recent overview of forward and inverse problems in nonlinear acoustics.

The objective of the present paper is to formalize some experimental observations and rigorously show that indeed the first-harmonic boundary data of the acoustic field contains enough information to determine the wave speed and the coefficient of diffusivity, and that the second-harmonic boundary data is sufficient to determine the coefficient of nonlinearity. In Section \ref{Section.ForwardProblem}, we introduce the mathematical framework to analyze the Westervelt equation for a boundary source harmonically oscillating at a fixed frequency. Then we show the existence and uniqueness of complex-valued time-periodic solutions in the form of a multi-harmonic Fourier series expansion. In Section 
\ref{Section.InverseProblem}, we state and prove the main mathematical result of this paper regarding the simultaneous determination of the wave speed, diffusivity and nonlinearity from knowledge of the first- and second-harmonic boundary data. Finally in Section \ref{Section.Discussion} we offer some remarks about the limitations of our work and possible extensions.

%%%%%%%%%%%%%%%%%%%%%%%%%%%%%%%%%%%
%%%%%%%%%  NEW SECTION  %%%%%%%%%%%
%%%%%%%%%%%%%%%%%%%%%%%%%%%%%%%%%%%
\section{Forward problem}
\label{Section.ForwardProblem}

In this section we establish the well-posedness of the forward problem for the Westervelt equation with time-harmonic boundary sources. We modify some arguments developed by Kaltenbacher, Lasiecka and Veljovic in \cite{Kaltenbacher2009,Kaltenbacher2011,Kaltenbacher2021} to construct time-periodic solutions to the Westervelt equation for a particular type of boundary condition. We consider a bounded connected domain $\Omega \subset \mathbb{R}^d$, for dimension $d=2,3$, with smooth boundary $\partial \Omega$. We consider a time-periodic Westervelt-type problem of the form
\begin{align}
    \mathcal{W}(u) = \partial_{t}^2 \left( u - \alpha u^2 \right) - \nabla \cdot ( \gamma \nabla u) - \nabla \cdot ( \beta  \nabla \partial_{t} u) = 0 \qquad &\text{in $(0,T) \times \Omega$} \label{eqn.001} \\
    ( \gamma + \beta \partial_{t} ) \partial_{\nu} u + \lambda \partial_{t} u + \eta u = g \, e^{-i \omega t} \qquad &\text{on $(0,T) \times \partial \Omega$} \label{eqn.002} \\
    u|_{t=0} = u|_{t=T} \qquad &\text{in $\Omega$} \label{eqn.003} \\
    \partial_{t} u|_{t=0} = \partial_{t} u|_{t=T} \qquad &\text{in $\Omega$} \label{eqn.004}
\end{align}
where $T=2\pi/\omega$ is the period, and $\omega$ is the angular frequency of oscillation for the boundary source (fundamental harmonic). In the boundary condition \eqref{eqn.002}, $\partial_{\nu}$ denotes the partial derivative in the outward normal direction, $\lambda=\lambda(x)$ and $\eta=\eta(x)$ are boundary coefficients, and the function $g : \partial \Omega \to \mathbb{C}$ represents the spatial dependence of the boundary source. The solution $u$ represents the acoustic pressure fluctuations, $\alpha = \alpha(x)$ is the coefficient of nonlinearity,
$\beta = \beta(x)$ is the diffusivity, and
$\gamma = \gamma(x)$ is the squared wave speed at constant hydrostatic pressure. The term $(u - \alpha u^2)$ in \eqref{eqn.001} is derived from a state equation relating the acoustic pressure and density fluctuations in fluids or soft tissues. That form is simply a second order Taylor approximation valid for relatively small pressure fluctuations. Hence, the model has an inherent limitation, namely, it only supports solutions satisfying $2 \alpha u < 1$. Otherwise, the equation degenerates and loses its properties. This mathematical limitation leads us to develop $L^\infty$ estimates by requiring the boundary source $g$ or the coefficient of nonlinearity $\alpha$ to be sufficiently small. Regarding all the media properties listed above, we make the following assumptions. Let the coefficient of nonlinearity $\alpha : \clo{\Omega} \to \mathbb{R}$, the coefficient of diffusivity $\beta : \clo{\Omega} \to \mathbb{R}$, the squared wave speed $\gamma : \clo{\Omega} \to \mathbb{R}$, 
and the boundary coefficients $\lambda : \partial \Omega \to \mathbb{R}$ and $\eta : \partial \Omega \to \mathbb{R}$ all be sufficiently regular and bounded functions (see precise statements below)  satisfying
\begin{align}
0 \leq \alpha(x), \quad \beta_0 \leq \beta(x), \quad \gamma_0 \leq \gamma(x), \quad \lambda_0 \leq \lambda(x), \quad \eta_{0} \leq \eta(x) \qquad \label{eqn.lemma_unique_lower_bounds}
\end{align}
for all $x$ in their domains for some constants $\beta_0, \gamma_0, \lambda_0, \eta_0 > 0$. The case $\alpha \equiv 0$ is not interesting. Therefore we also assume that the coefficient of nonlinearity is not vanishing.

The strategy is to build a solution for \eqref{eqn.001}-\eqref{eqn.004} from the following Fourier ansatz,
\begin{align}
    u^{K}(t,x) = \sum_{k=0}^{K} u_{k}(x) e^{-i k \omega t} \label{eqn.Ansatz}
\end{align}
where the terms $u_{k}$ are defined recursively and then show that in the limit as $K \to \infty$, the sequence $\{ u_{K} \}$ converges in the appropriate normed space to a function $u$ that satisfies \eqref{eqn.001}-\eqref{eqn.004}. Note that by construction, the ansatz \eqref{eqn.Ansatz} is $T$-periodic, so it satisfies \eqref{eqn.003}-\eqref{eqn.004}.

%%%%%%%%%%%%%%%%%%%%%%%%%%%%%%%%%%%%%%%%%%
\subsection{Recursive definition for the multi-harmonic ansatz}
We will make repeated use the of Cauchy product formula for two absolutely converging series,
\begin{align}
\left( \sum_{j=0}^{\infty} a_{j} \right) \left( \sum_{k=0}^{\infty} b_{k} \right) = \sum_{k=0}^{\infty} \sum_{l=0}^{k}  a_{l} b_{k-l} \label{eqn.Cauchy}
\end{align}
to express all the terms from the application of the Westervelt operator $\mathcal{W}$ to the limit of the ansatz \eqref{eqn.Ansatz}. So provided that \eqref{eqn.Ansatz} converges absolutely as $K \to \infty$, we would obtain
\begin{align}
\partial_{t} u &= - \sum_{k=0}^{\infty} i k \omega u_{k}  e^{-i k \omega t}. \label{eqn.099} \\
\partial_{t}^{2} u &= - \sum_{k=0}^{\infty} (k \omega)^2 u_{k}  e^{-i k \omega t}. \label{eqn.100} \\
(\partial_{t}^{2} u) u &= - \sum_{k=0}^{\infty} \left( \sum_{l=0}^{k} \omega^2 l^2  u_l u_{k-l} \right) e^{-i k \omega t}, \label{eqn.101} \\
(\partial_{t} u)^2 &= - \sum_{k=0}^{\infty} \left( \sum_{l=0}^{k} \omega^2 l (k-l)  u_l u_{k-l} \right) e^{-i k \omega t}, \label{eqn.102} \\
\nabla \cdot ( \gamma \nabla u ) &= \sum_{k=0}^{\infty} \nabla \cdot ( \gamma \nabla u_{k} )  e^{-i k \omega t}, \label{eqn.103} \\
\nabla \cdot ( \beta \partial_{t} \nabla u ) &= -\sum_{k=0}^{\infty} i k \omega \nabla \cdot ( \beta \nabla u_k )  e^{-i k \omega t}. \label{eqn.104}
\end{align}
Note that \eqref{eqn.101} and \eqref{eqn.102} can be combined to obtain 
\begin{align}
\partial_t^2 ( u^2) = 2 \left( (\partial_{t}^{2} u) u + (\partial_{t} u)^2 \right) = - 2\sum_{k=0}^{\infty} \left( \sum_{l=0}^{k} \omega^2 l \,  k \,  u_l u_{k-l} \right) e^{-i k \omega t} .\label{eqn.105}
\end{align}
Therefore, if the limit of the ansatz existed and were to solve \eqref{eqn.001}, using the orthogonality of the Fourier basis $\{ e^{- i k \omega t} \}_{k=0}^{\infty}$, we would necessarily arrive at the following result for the zeroth term,
\begin{align}
u_0 \equiv 0,&  \label{eqn.110}
\end{align}
due to the boundary condition \eqref{eqn.002} and the condition \eqref{eqn.lemma_unique_lower_bounds}, and for the fundamental harmonic,
\begin{equation} \label{eqn.111}
\begin{aligned}
\omega^2 u_1 + \nabla \cdot  (\gamma - i \omega \beta) \nabla u_1  &= 0, \qquad \text{in $\Omega$}, \\
(\gamma - i \omega \beta) \partial_{\nu} u_1 - (i \omega \lambda - \eta) u_1  &= g, \qquad 
 \text{on $\partial \Omega$}, 
\end{aligned}
\end{equation}
and a recursive formula for the coefficient of the $k^{\rm th}$-order harmonic ($k \geq 2$),
\begin{equation} \label{eqn.112}
\begin{aligned}
(k \omega)^2 u_k + \nabla \cdot ( \gamma - i k \omega \beta) \nabla u_k &=  2 \alpha \, (k \omega)^2 \sum_{l=1}^{k-1} \frac{l}{k}  \,  u_l u_{k-l}, \qquad  \text{in $\Omega$},  \\
( \gamma - i k \omega \beta ) \partial_{\nu} u_k - (i k \omega \lambda - \eta) u_k &= 0, \qquad  \text{on $\partial \Omega$.} 
\end{aligned}
\end{equation}

In other words, $u_1$ satisfies the linear Helmholtz equation in $\Omega$ being forced by $g$ at the boundary $\partial \Omega$. The coefficient of nonlinearity $\alpha$ plays no role in the equation for $u_1$. 
All the other Fourier components satisfy linear Helmholtz equations of increasing frequencies, forced by products of previous terms, and subject to homogeneous boundary conditions.

%%%%%%%%%%%%%%%%%%%%%%%%%%%%%%%%%%%%%%%%%%
\subsection{Existence of a solution using the multi-harmonic ansatz}

Here and in the rest of the paper, $H^{s}(D)$ for $s \geq 0$ denotes the Sobolev Hilbert space of functions defined in the domain $D = \Omega$ or $D = \partial \Omega$ having at least $s$ square integrable weak derivatives. Similarly, $W^{s,\infty}(D)$ denotes the Sobolev space of functions defined in the domain $D$ having at least $s$ essentially bounded weak derivatives. We also use the common notation $L^2 = H^0$ and $L^{\infty} = W^{0,\infty}$. Detailed definitions are found in \cite{McLean2000,EvansPDE,Gil-Tru-2001}.

Using the recursive definition \eqref{eqn.111}-\eqref{eqn.112} of the ansatz, we proceed to show its convergence. The main tool is the following frequency-explicit stability estimates for the Helmholtz equation. First, we need a lemma to estimate the norms of the solution $u_1$ in terms of the regularity and size of the boundary source $g$.

\begin{lemma} \label{Thm.Lemma_v1}
Assume that $\beta, \gamma \in W^{1,\infty}(\Omega)$ and $\lambda, \eta \in L^{\infty}(\partial \Omega)$ are real-valued functions satisfying \eqref{eqn.lemma_unique_lower_bounds}. Let $g \in H^{0}(\partial \Omega)$. Suppose that $v \in H^1(\Omega)$ is the weak solution to the following Helmholtz problem,
\begin{equation} \label{eqn.252} 
\begin{aligned}
\nabla \cdot  \mu \nabla v  + \omega^2 v &= 0 \qquad \text{in $\Omega$,}  \\
\mu \partial_{\nu}v - (i \omega \lambda - \eta) v &=  g \qquad \text{on $\partial \Omega$},
\end{aligned}
\end{equation}
where
\begin{align}
\mu =\gamma - i \omega \beta.
\label{eqn.mu_1} 
\end{align}

Then the following estimates hold
\begin{align} 
\omega \| v \|_{H^0(\partial \Omega)} &\leq C \| g \|_{H^{0}(\partial \Omega)} \label{eqn.C0} \\
\omega^{3/2} \| v \|_{H^0(\Omega)} &\leq C \| g \|_{H^{0}(\partial \Omega)} \label{eqn.C1}  \\
\omega \| \nabla v \|_{H^0(\Omega)} &\leq C \| g \|_{H^{0}(\partial \Omega)} \label{eqn.C2} \\
\omega^{1/2} \| \Delta v \|_{H^0(\Omega)} &\leq  C  \| g \|_{H^0(\partial \Omega)}  \label{eqn.C3}
\end{align}
for some constant $C=C(\beta, \gamma, \lambda, \eta)$ independent of $\omega \geq \omega_{o}$ for some fixed $\omega_{o}$. 

Moreover, if $g \in H^{s}(\partial \Omega)$ for $s \geq 1/2$ and $\Omega \subset \mathbb{R}^d$ for $d=2,3$, then $v \in H^2(\Omega) \subset L^{\infty}(\Omega)$ and 
\begin{align} 
\| v \|_{L^{\infty}(\Omega)} \lesssim \| v \|_{H^{2}(\Omega)}  \leq \frac{C}{\omega^{1/2}} \| g \|_{H^{s}(\partial \Omega)} \label{eqn.C4}
\end{align}
for some constant $C=C(\Omega, \beta, \gamma, \lambda, \eta)$ independent of $\omega \geq \omega_{o}$ for some fixed $\omega_{o}$.
\end{lemma}

\begin{proof}
The boundary value problem \eqref{eqn.252} is well-posed in weak and strong formulations for all real $\omega$ due to the positivity condition \eqref{eqn.lemma_unique_lower_bounds} on the dissipative terms $\beta$ and $\lambda$. See details in \cite[Ch. 4]{McLean2000}, \cite[Ch. 2]{Lio-Mag-Book-1972}. Take $v$, the weak solution to \eqref{eqn.252}, and consider weak bilinear form applied to $v$ and $\clo{v}$,
\begin{align}
- \int_{\Omega} \mu | \nabla v |^2 dx +  \int_{\Omega} \omega^2 |v|^2 dx + \int_{\partial \Omega} (i \omega \lambda - \eta) |v|^2 \, dS(x) = - \int_{\partial \Omega} g \clo{v} \, dS(x).  \label{eqn.bilinear_u1}
\end{align}
The imaginary part of \eqref{eqn.bilinear_u1}  leads to
\begin{align}
\omega \int_{\Omega} \beta | \nabla v |^2 dx  + \omega  \int_{\partial \Omega} \lambda |v|^2 dS(x) = - \Im \int_{\partial \Omega} g \clo{v} dS(x) \nonumber
\end{align}
which, after the application of the Cauchy-Schwarz inequality, implies that
\begin{align}
\omega \| \nabla v \|_{H^0(\Omega)}^2  + \omega \| v \|_{H^0(\partial \Omega)}^2 \leq C \| g \|_{H^{0}(\partial \Omega)}  \| v \|_{H^{0}(\partial \Omega)} \nonumber
\end{align}
which renders \eqref{eqn.C0} and \eqref{eqn.C2} for some constant $C=C(\beta, \lambda)$ independent of $\omega > 0$. 

The real part of the bilinear form \eqref{eqn.bilinear_u1} leads to
\begin{align}
- \int_{\Omega} \gamma |\nabla v|^2 dx + \omega^2 \int_{\Omega} |v|^2 dx - \int_{\partial \Omega} \eta |v|^2 dS(x)  = - \Re \int_{\partial \Omega} g \clo{v} \, dS(x) \nonumber
\end{align}
which, after using the Cauchy-Schwarz inequality and the previous results, implies that
\begin{align}
\omega^2 \|  v \|_{H^0(\Omega)}^2  &\leq C \left(  \| \nabla v \|_{H^0(\Omega)}^2 + \| v \|_{H^0(\partial \Omega)}^2 + \| g \|_{H^0(\partial \Omega)} \| v \|_{H^0(\partial \Omega)} \right) \leq C \left( \frac{1}{\omega^2} + \frac{1}{\omega} \right) \| g \|_{H^0(\partial \Omega)}^2 \nonumber
\end{align}
thus, rendering \eqref{eqn.C1} for some other constant $C=C(\beta, \gamma, \lambda, \eta)$ provided that $\omega \geq \omega_{o}$ for some fixed $\omega_o$.
Now we take a sequence $\{ g_{m} \}$ of smooth functions converging to $g$ in the norm of $H^{0}(\partial \Omega)$. These boundary sources induce strong solutions $\{ v_{m} \}$ to \eqref{eqn.252}. Therefore, simply taking the $H^0(\Omega)$-norm of \eqref{eqn.252} leads to
\begin{align}
\| \Delta v_m \|_{H^0(\Omega)} &\leq C \left( \| (\nabla \mu \cdot \nabla v_{m})/\mu \|_{H^{0}(\Omega)} + \| \omega^2  v_m / \mu \|_{H^0(\Omega)} \right) \leq C \left( \frac{1}{\omega} + \frac{ 1 }{ \omega^{1/2} } \right) \| g_m \|_{H^0(\partial \Omega)}    \nonumber
\end{align}
Hence, we obtain the desired result \eqref{eqn.C3} in the limit as $m \to \infty$ for a constant $C=C(\beta, \gamma, \lambda, \eta)$ independent of $\omega \geq \omega_{o}$.

Finally, by the trace theorem, a weak solution $v \in H^{1}(\Omega)$ possesses a Dirichlet trace in $H^{1/2}(\partial \Omega)$. Therefore, assuming that $g \in H^{s}(\partial \Omega)$ for $s \geq 1/2$, the weak solution $v$ possesses a Neumann trace in $H^{1/2}(\partial \Omega)$. Therefore, $v$ is a strong solution to \eqref{eqn.252}. See \cite[Ch.4]{McLean2000}, \cite[Ch.6]{EvansPDE}, or \cite[Ch.8]{Gil-Tru-2001} for details. Therefore, $v$ satisfies a Poisson equation of the type $\Delta v = F$ with a source $F \in H^{0}(\Omega)$ such that $\| F \|_{H^{0}(\Omega)} \leq C  \omega^{-1/2} \| g \|_{H^{0}(\partial \Omega)}$, and boundary condition $\partial_{\nu} v  - (i \omega \lambda - \eta) v /\mu = g/\mu$, where $|(i \omega \lambda - \eta) / \mu| \leq \mathcal{O}(1)$ as $\omega \to \infty$. 
For such a function $v \in H^2(\Omega)$, we have 
\begin{align}
\| v \|_{H^2(\Omega)} \leq C \left( \| F \|_{H^0(\Omega)} + \| g / \mu \|_{H^{1/2}(\partial \Omega)} \right) \leq C \left( \frac{1}{\omega^{1/2}} + \frac{1}{\omega} \right) \| g \|_{H^{1/2}(\partial \Omega)}  \leq \frac{C} {\omega^{1/2}} \| g \|_{H^{s}(\partial \Omega)} \nonumber
\end{align}
for some constant $C=C(\Omega, \beta, \gamma, \lambda, \eta)$ independent of $\omega \geq \omega_{o}$. Now since $H^{2}(\Omega)$ is continuously embedded in $L^{\infty}(\Omega)$ for a bounded domain $\Omega \subset \mathbb{R}^d$ with smooth boundary $\partial \Omega$ for dimension $d=2,3$ \cite{McLean2000, EvansPDE, Gil-Tru-2001}, we obtain \eqref{eqn.C4}.
\end{proof}

A second lemma is needed to estimate the norms of the recursive functions $u_k$ in terms of the regularity and size of a generic interior source $f$.

\begin{lemma} \label{Thm.Lemma_vk}
Assume that $\alpha \in L^{\infty}(\Omega)$, $\beta, \gamma \in W^{1,\infty}(\Omega)$ and $\lambda, \eta \in L^{\infty}(\partial \Omega)$ are real-valued functions satisfying \eqref{eqn.lemma_unique_lower_bounds}. Let $f \in H^{0}(\Omega)$. Suppose that $v \in H^1(\Omega)$ is the weak solution to the following Helmholtz problem
\begin{equation} \label{eqn.202}
\begin{aligned}
\nabla \cdot  \mu_{k} \nabla v + (k \omega )^2  v &= (k \omega)^2 \, 2 \alpha f \qquad \text{in $\Omega$,}  \\
\mu_{k} \partial_{\nu}v - (i k \omega \lambda - \eta) v &= 0 \qquad \text{on $\partial \Omega$},
\end{aligned}
\end{equation}
where 
\begin{align}
\mu_{k} = \gamma - i k \omega \beta.
\label{eqn.mu_rho} 
\end{align}
Then there is a constant $C=C(\Omega,\beta , \gamma, \lambda, \eta)$ independent of $k\geq 2$ and of $\omega \geq \omega_{o}$ for some fixed $\omega_{o}$, such that,
\begin{align} 
\|v\|_{H^{0}(\partial \Omega)} &\leq C (k \omega)^{1/2} \| \alpha f \|_{H^{0}(\Omega)} \label{eqn.D1} \\
\|v\|_{H^{0}(\Omega)} &\leq C \| \alpha f \|_{H^{0}(\Omega)} \label{eqn.D0} \\
\|\nabla v\|_{H^{0}(\Omega)} &\leq C (k \omega)^{1/2} \|  \alpha f \|_{H^{0}(\Omega)} \label{eqn.D2} \\
\|\Delta v\|_{H^{0}(\Omega)} &\leq C k \omega \| \alpha f \|_{H^{0}(\Omega)} \label{eqn.D3}
\end{align}
Moreover, if $\Omega \subset \mathbb{R}^d$ for $d=2,3$, then  $v \in H^2(\Omega) \subset L^{\infty}(\Omega)$ and 
\begin{align}
\| v \|_{L^{\infty}(\Omega)} \lesssim \| v \|_{H^{2}(\Omega)} \leq C k \omega \, \| \alpha f \|_{H^{0}(\Omega)} \label{eqn.D4}
\end{align}
for some other constant $C=C(\Omega, \beta, \gamma, \delta, \lambda)$ independent of $\omega \geq \omega_{o}$ and of $k \geq 2$.
\end{lemma}

\begin{proof}
Since $v$ is the weak solution to \eqref{eqn.202}, the weak bilinear form applied to $v$ and $\clo{v}$ reads
\begin{align}
- \int_{\Omega} \mu_{k} | \nabla v |^2 dx + \int_{\Omega} (k \omega )^2  |v|^2 dx + \int_{\partial \Omega} (i k \omega \lambda -\eta) |v|^2 \, dS(x) = (k \omega)^2 \int_{\Omega}2 \alpha  f \clo{v} \, dx. \label{eqn.bilinear_uk}
\end{align}

The imaginary part of the form \eqref{eqn.bilinear_uk} leads to
\begin{align}
k \omega \int_{\Omega} \beta |\nabla v|^2 dx  + k \omega \int_{\partial \Omega} \lambda  |v|^2 dS(x) = (k \omega)^2 \Im \int_{\Omega} 2 \alpha f \clo{v} dx \nonumber
\end{align}
which, after the application of the Cauchy-Schwarz inequality, implies that
\begin{align}
 \| \nabla v \|_{H^{0}(\Omega)}^{2} \leq C k \omega  \| \alpha f \|_{H^{0}(\Omega)}  \| v \|_{H^{0}(\Omega)} \nonumber
\end{align}
for some constant $C=C(\beta)$ valid for all $\omega > 0$ and all $k \geq 2$. Additionally, 
\begin{align}
\| v \|_{H^{0}(\partial \Omega)}^2 \leq C k \omega \| \alpha f \|_{H^{0}(\Omega)}  \| v \|_{H^{0}(\Omega)} \nonumber
\end{align}
for some constant $C=C(\lambda)$, valid for all $\omega > 0$ and all $k \geq 2$.

Now, the real part of the form \eqref{eqn.bilinear_uk} leads to
\begin{align}
- \int_{\Omega} \gamma |\nabla v|^2 dx + (k \omega)^2 \int_{\Omega} |v|^2 dx  - \int_{\partial \Omega} \eta |v|^2 dS(x) =  (k \omega)^2 \Re \int_{\Omega} 2 \alpha f \clo{v} dx  \nonumber
\end{align}
which, using the previous results and the Cauchy-Schwarz inequality, implies that
\begin{align*}
(k \omega)^2 \| v \|_{H^{0}(\Omega)}^2 &\leq C \left( \| \nabla v \|_{H^{0}(\Omega)}^2 + \| v \|_{H^{0}(\partial \Omega)}^2 + (k \omega)^2 \| \alpha f \|_{H^{0}(\Omega)} \| v \|_{H^{0}(\Omega)} \right) \\ &\leq C \left( (k \omega) + (k \omega)^2 \right) \| \alpha f \|_{H^{0}(\Omega)} \| v \|_{H^{0}(\Omega)} 
\end{align*}
which renders \eqref{eqn.D0} for some other constant $C=C(\beta, \gamma, \lambda, \eta)$ independent of $\omega \geq \omega_{o}$ and $k \geq 2$. Plugging \eqref{eqn.D0} back into the previous inequalities renders \eqref{eqn.D1} and \eqref{eqn.D2}. 

Finally, since $f \in H^0(\Omega)$ and $\alpha \in W^{0,\infty}(\Omega)$, regularity theory for elliptic equations implies that $v \in H^{2}(\Omega)$ is a strong solution to \eqref{eqn.202} \cite{McLean2000,EvansPDE,Gil-Tru-2001}. Therefore, simply taking the $H^{0}(\Omega)$-norm of the differential equation \eqref{eqn.202} leads to 
\begin{align}
\| \Delta v \|_{H^{0}(\Omega)} \leq C \left(  \| (\nabla \mu_{k} \cdot \nabla v) / \mu_{k} \|_{H^{0}(\Omega)} + \| (k \omega)^2 v / \mu_{k} \|_{H^{0}(\Omega)} +  (k \omega)^2  \| \alpha f / \mu_{k} \|_{H^{0}(\Omega)} \right) \nonumber
\end{align}
which yields \eqref{eqn.D3} after the use of the previous inequalities. Due to the boundary condition in \eqref{eqn.202}, and a bootstrapping argument between the regularity of the Neumann and Dirichlet boundary traces, the solution $v$ has a sufficiently regular trace to obtain $\| v \|_{H^{2}(\Omega)} \leq C \| \Delta v \|_{H^{0}(\Omega)} $ for a constant $C=C(\Omega)$. See \cite[Ch.4]{McLean2000}, \cite[Ch.6]{EvansPDE}, or \cite[Ch.8]{Gil-Tru-2001} for details. Finally, since $H^2(\Omega)$ is continuously embedded in $L^{\infty}(\Omega)$ for a smooth domain $\Omega \subset \mathbb{R}^d$ in dimension $d=2,3$ \cite{McLean2000,EvansPDE,Gil-Tru-2001}, we obtain \eqref{eqn.D4}.
\end{proof}

Now we proceed to apply Lemmas \ref{Thm.Lemma_v1} and \ref{Thm.Lemma_vk} to quantify the growth/decay of the terms in the ansatz \eqref{eqn.Ansatz}.

\begin{lemma} \label{Thm.Lemma_u_decay}
Assume that $\alpha \in L^{\infty}(\Omega)$, $\beta, \gamma \in W^{1,\infty}(\Omega)$ and $\lambda, \eta \in L^{\infty}(\partial \Omega)$ are real-valued functions satisfying \eqref{eqn.lemma_unique_lower_bounds}.
Let the sequence $\{ u_{k} \}$ be defined recursively as the weak solutions to the systems \eqref{eqn.111}-\eqref{eqn.112}. If $g \in H^{s}(\partial \Omega)$ for $s \geq 1/2$ and $\Omega \subset \mathbb{R}^d$ for $d=2,3$ is a bounded connected domain with smooth boundary $\partial \Omega$, then 
\begin{itemize}
\item[(a)] If $\| \alpha \|_{L^{\infty}(\Omega)} > 0$, the functions $u_{k} \in H^{2}(\Omega) \subset L^{\infty}(\Omega)$ and satisfy
\begin{align}
\| u_k\|_{H^{0}(\Omega)} & \leq \frac{h_k}{C k \omega \| \alpha \|_{L^{\infty}(\Omega)} } \left( \frac{C^2}{\omega^{1/2}} \| g \|_{H^{s}(\partial \Omega)} \| \alpha \|_{L^{\infty}(\Omega)} \right)^k \label{eqn.normH0_uk} \\
\| u_k\|_{H^{2}(\Omega)} &\leq  \frac{h_k}{C \| \alpha \|_{L^{\infty}(\Omega)}} \left( \frac{C^2}{\omega^{1/2}} \| g \|_{H^{s}(\partial \Omega)} \| \alpha \|_{L^{\infty}(\Omega)} \right)^k \label{eqn.normH2_uk} \\
\| u_k\|_{L^{\infty}(\Omega)} &\leq  \frac{h_k}{C \| \alpha \|_{L^{\infty}(\Omega)}} \left( \frac{C^2}{\omega^{1/2}} \| g \|_{H^{s}(\partial \Omega)} \| \alpha \|_{L^{\infty}(\Omega)} \right)^k \label{eqn.normLinfty_uk}
\end{align}
for some constant $C=C(\Omega,\beta,\gamma,\lambda,\eta)$ independent of $\omega \geq \omega_{o}$ and of $k \geq 1$, where the coefficients $h_k$ are defined recursively as follows
\begin{align}
h_1 = 1, \qquad \text{and} \quad h_k = \sum_{l=1}^{k-1} h_{l} h_{k-l}, \qquad \text{for $k\geq 2$}. \label{eqn.h_seq}
\end{align}

\item[(b)] 
The sequence $h_{k}$ defined in \eqref{eqn.h_seq} satisfies
\begin{align}
h_k \leq \frac{5^{k-1}}{k^2} \qquad \text{for all $k \geq 1$.} \label{eqn.h_seq03}
\end{align}
\end{itemize}
\end{lemma}

\begin{proof}
We will rely on Lemmas \ref{Thm.Lemma_v1} and \ref{Thm.Lemma_vk}. In this proof, the constant $C=C(\Omega,\beta,\gamma,\lambda,\eta)$ may change from inequality to inequality, but it is carefully kept independent of the frequency $\omega \geq \omega_{o}$ and of $k \geq 1$ as it was done in  Lemmas \ref{Thm.Lemma_v1} and \ref{Thm.Lemma_vk}.

For part (a), the claim is clearly true for $k=1$ as shown by Lemma \ref{Thm.Lemma_v1}. For $k \geq 2$, we proceed by induction. Suppose that \eqref{eqn.normH0_uk}-\eqref{eqn.normLinfty_uk} are true for $l=1, \ldots ,k$. We rely on Lemma \ref{Thm.Lemma_vk} which was formulated for a generic source $f$. So in order to match \eqref{eqn.112}, now we apply it for a specific source of the form
\begin{align}
f_{k+1} =  \sum_{l=1}^{k} \frac{l}{k+1} \,  u_{l} \, u_{k+1-l} \label{eqn.source_k}
\end{align}
and we obtain
\begin{align*}
\| u_{k+1} \|_{H^{0}(\Omega)} &\leq C \| \alpha \|_{L^{\infty}(\Omega)} \| f_{k+1} \|_{H^{0}(\Omega)} 
 \leq \frac{C \| \alpha \|_{L^{\infty}(\Omega)} }{k+1} \sum_{l=1}^{k} l \| u_{l} u_{k+1-l} \|_{H^{0}(\Omega)} \\ &\leq \frac{C  \| \alpha \|_{L^{\infty}(\Omega)} }{k+1} \sum_{l=1}^{k} l \| u_{l} \|_{H^{0}(\Omega)}  \| u_{k+1-l} \|_{L^{\infty}(\Omega)}   \\
 &\leq \frac{C  \| \alpha \|_{L^{\infty}(\Omega)} }{k+1} \sum_{l=1}^{k}  \frac{l \, h_l}{C l \omega  \| \alpha \|_{L^{\infty}(\Omega)}}   \frac{h_{k+1-l}}{C  \| \alpha \|_{L^{\infty}(\Omega)}}
 \left( \frac{C^2}{\omega^{1/2}} \| g \|_{H^{s}(\partial \Omega)} \| \alpha \|_{L^{\infty}(\Omega)} \right)^{k+1} 
  \\ 
 & = \frac{h_{k+1}}{C (k+1) \omega \| \alpha \|_{L^{\infty}(\Omega)} } \left( \frac{C^2 }{\omega^{1/2}} \| g \|_{H^{s}(\partial \Omega)} \| \alpha \|_{L^{\infty}(\Omega)} \right)^{k+1} 
\end{align*}
which renders \eqref{eqn.normH0_uk} as desired. Similarly to obtain \eqref{eqn.normH2_uk}, from Lemma \ref{Thm.Lemma_vk} we have that
\begin{align*}
\| u_{k+1} \|_{H^{2}(\Omega)} &\leq C (k+1) \omega \| \alpha \|_{L^{\infty}(\Omega)} \| f_{k+1} \|_{H^{0}(\Omega)} \\
&\leq C \omega \| \alpha \|_{L^{\infty}(\Omega)}  \sum_{l=1}^{k} l \| u_{l} u_{k+1-l} \|_{H^{0}(\Omega)} \\ 
 &\leq C \omega \| \alpha \|_{L^{\infty}(\Omega)} \sum_{l=1}^{k} l \| u_{l} \|_{H^{0}(\Omega)}  \| u_{k+1-l} \|_{L^{\infty}(\Omega)}  \nonumber \\
  &\leq C \omega \| \alpha \|_{L^{\infty}(\Omega)} \sum_{l=1}^{k} \frac{ l \, h_l}{C l \omega \| \alpha \|_{L^{\infty}(\Omega)}} \frac{h_{k+1-l}}{C \| \alpha \|_{L^{\infty}(\Omega)}}
 \left( \frac{C^2}{\omega^{1/2}} \| g \|_{H^{s}(\partial \Omega)} \| \alpha \|_{L^{\infty}(\Omega)} \right)^{k+1} 
 \nonumber \\ 
 &= \frac{h_{k+1}}{C \| \alpha \|_{L^{\infty}(\Omega)}} \left( \frac{C^2 }{\omega^{1/2}} \| g \|_{H^{s}(\partial \Omega)} \| \alpha \|_{L^{\infty}(\Omega)} \right)^{k+1} \nonumber
\end{align*}
which renders \eqref{eqn.normH2_uk}. Here we have used the definition \eqref{eqn.h_seq}. The estimate \eqref{eqn.normLinfty_uk} is obtained using exactly the same argument. 

For part (b), we proceed to prove first that there exists constants $\theta > 0$ and $r > 0$ such that 
\begin{align}
h_k \leq \frac{ \theta \, r^{k}}{k^2} \qquad \text{for all $k \geq 1$.} \label{eqn.h_seq03alt}
\end{align}
We do this by induction once more. Suppose that $h_l \leq \theta \, r^{l} / l^2$ for $l=1, 2 \ldots, k$. By the definition \eqref{eqn.h_seq}, we have that
\begin{align}
h_{k+1} = \sum_{l=1}^{k} h_{l} h_{k+1-l} \leq \theta^2 \sum_{l=1}^{k} \frac{r^{l}}{ l^2 } \frac{ r^{k+1-l} }{ (k+1-l)^2 } \leq \theta^2 r^{k+1} \sum_{l=1}^{k} \frac{1}{ l^2 (k+1-l)^2 }. \label{eqn.h_seq05}
\end{align}
Now, using an integral to bound the sum, we obtain
\begin{align}
\sum_{l=1}^{k} \frac{1}{ l^2 (k+1-l)^2 } \leq 2 \left( \frac{1}{k^2} + \int_{1}^{(k+1)/2} \frac{dy}{y^2 (k+1-y)^2} \right) = 2 \left( 
\frac{1}{k^2} + \frac{k^2 + 2 k \ln k - 1}{k (k+1)^3} \right) \label{eqn.h_seq07}
\end{align}
valid for $k \geq 2$ and where we have used the symmetry of the terms in the sum about $l_{\rm mid} = (k+1)/2$. Plugging \eqref{eqn.h_seq07} back into \eqref{eqn.h_seq05} and re-arranging some factors, we obtain
\begin{align}
h_{k+1} \leq  \frac{\theta \, r^{k+1}}{ (k+1)^2 } \, \theta  \,\varrho_{k}, \qquad \text{where} \quad \varrho_{k} = 2 \left( 
\frac{(k+1)^2}{k^2} + \frac{k^2 + 2 k \ln k - 1}{k (k+1)} \right). \label{eqn.h_seq09}
\end{align}
Note that $\varrho_{k}$ (independent of $\theta$ and of $r$) is a decreasing function of $k$, and $\varrho_{k} \to 4$ as $k \to \infty$. 
Hence, it only remains to choose $\theta$ small enough so that $\theta \varrho_{k_{*}} \leq 1$ and simultaneously $r$ large enough to ensure that \eqref{eqn.h_seq03alt} is satisfied for $k=1,2, \ldots, k_{*}$ for some fixed $k_{*}$. For instance, for $k_{*} = 2$, then $\varrho_{2} \approx 6.42 \leq 7$, and choosing $\theta = 1/7$ and $r = 7$ suffices to establish the inductive hypothesis. Similarly for $k_{*} = 3$, then $\varrho_{3} \approx 5.99 \leq 6$, and choosing $\theta = 1/6$ and $r = 6$ suffices to establish the inductive hypothesis. After checking some more values, we also find that for $k_{*} = 11$, we have $\varrho_{11} \leq 5$. In this case it suffices to choose $\theta=1/5$ and $r=5$ for \eqref{eqn.h_seq03} to be satisfied for $k=1,2, \ldots, 11$. The induction shown above concludes the proof.
\end{proof}

With the aid of the previous Lemmas we are ready to show that the ansatz \eqref{eqn.Ansatz} converges to a solution of the Westervelt equation.

\begin{theorem}[Existence] \label{thm.wellpose_forward}
Let $\Omega \subset \mathbb{R}^d$ for $d=2,3$ be a bounded connected domain with smooth boundary $\partial \Omega$. 
Assume that $\alpha \in L^{\infty}(\Omega)$, $\beta, \gamma \in W^{1,\infty}(\Omega)$ and $\lambda, \eta \in L^{\infty}(\partial \Omega)$ are real-valued functions satisfying \eqref{eqn.lemma_unique_lower_bounds}, and $\alpha$ is non vanishing. Let the sequence $\{ u_{k} \}$ be defined recursively as the weak solutions to the systems \eqref{eqn.111}-\eqref{eqn.112}.
If $g \in H^{s}(\partial \Omega)$ for $s \geq 1/2$ and $\omega$ sufficiently large such that
\begin{align}
r = \frac{5 C^2 \| \alpha \|_{L^{\infty}(\Omega)} \| g \|_{H^{s}(\partial \Omega)}}{\omega^{1/2}} < 1 \label{eqn.small_g}
\end{align}
where $C=C(\Omega,\beta, \gamma,\lambda,\eta)$ is the constant appearing on Lemma \ref{Thm.Lemma_u_decay}, then the ansatz \eqref{eqn.Ansatz} converges absolutely in $C^{2}([0,T];H^{2}(\Omega))$ as $K \to \infty$, and the limit $u$ solves the Westervelt system \eqref{eqn.001}-\eqref{eqn.004}, satisfying the following stability estimates
\begin{align}
\| \partial_{t}^m u \|_{H^{0}((0,T); H^{2}(\Omega))} &\leq C \omega^{m-1} \frac{ (2\pi)^{1/2} \| g \|_{H^{s}(\partial \Omega)} }{ (1-r) } , \qquad m=0,1,2, \label{eqn.stability} \\
\| \partial_{t}^m u \|_{C^{0}([0,T]; H^{2}(\Omega))} &\leq C \omega^{m-1/2} \frac{ \| g \|_{H^{s}(\partial \Omega)} }{ (1-r) } , \qquad m=0,1,2, \label{eqn.stability2}
\end{align}
where $T = 2 \pi / \omega$ and $C$ is independent of $\omega \geq \omega_{o}$ for some sufficiently large $\omega_o$.
\end{theorem}

\begin{proof}
We show the convergence of the series  $\{ u^{K} \}$ given by \eqref{eqn.Ansatz} in the norm of $H^{0}((0,T); H^{2}(\Omega))$.
Most of the needed steps are already proven in Lemmas \ref{Thm.Lemma_v1}-\ref{Thm.Lemma_u_decay} which straightforwardly imply that
\begin{align*}
\sum_{k=0}^{K} \| u_{k} e^{-i k \omega t} \|_{H^0((0,T); H^{2}(\Omega))} &\leq T^{1/2} \sum_{k=1}^{K} \| u_{k} \|_{H^2(\Omega)} \\
&\leq \frac{T^{1/2}}{C \| \alpha \|_{L^{\infty}(\Omega)}}  \sum_{k=1}^{K} h_k  \left( \frac{C^2}{\omega^{1/2}} \| g \|_{H^{s}(\partial \Omega)} \| \alpha \|_{L^{\infty}(\Omega)} \right)^k \\
&\leq \frac{T^{1/2}}{5 C \| \alpha \|_{L^{\infty}(\Omega)}}  \sum_{k=1}^{K}  \frac{1}{k^2} \left( \frac{ 5 C^2}{\omega^{1/2}} \| g \|_{H^{s}(\partial \Omega)} \| \alpha \|_{L^{\infty}(\Omega)} \right)^k \\
&\leq \frac{T^{1/2}}{5 C \| \alpha \|_{L^{\infty}(\Omega)}}  \sum_{k=1}^{\infty} r^k = \frac{T^{1/2}}{5 C \| \alpha \|_{L^{\infty}(\Omega)}} \frac{r}{1-r} 
\end{align*}
for all $K \geq 1$, provided that $r < 1$ as assumed in \eqref{eqn.small_g}. Therefore we have absolute convergence of the partial sums $\{ u^{K} \}$ in the Hilbert space $H^{0}((0,T); H^{2}(\Omega))$. Due to completeness, the sequence $\{ u^{K} \}$ converges to an element $u \in H^{0}((0,T); H^{2}(\Omega))$. 
Very similar arguments show convergence in 
$H^{2}((0,T); H^{2}(\Omega))$ and
$C^{2}([0,T]; H^{2}(\Omega))$ as well.
Therefore the stability estimates \eqref{eqn.stability}-\eqref{eqn.stability2} are easily verified. 

It only remains to show that $u$ solves the Westervelt system \eqref{eqn.001}-\eqref{eqn.004}. First note that $u^K$ is periodic for every $K \geq 1$ and convergence in $C^{2}([0,T]; H^{2}(\Omega))$ implies that the limit $u$ satisfies the periodicity conditions \eqref{eqn.003}-\eqref{eqn.004} in the norm of $H^2(\Omega)$. Now, take $v \in L^2(\Omega)$ and $k \in \mathbb{N}$ arbitrary. Multiply \eqref{eqn.001} by $v e^{- i k \omega t}$, integrate over $\Omega \times (0,T)$, and use the orthogonality of the Fourier basis for the time interval $(0,T)$ to obtain
\begin{align}
( \mathcal{W}(u) , v e^{- i k \omega t} )_{\Omega \times (0,T)} &= - (k \omega)^2 ( u_k - 2 \alpha \sum_{l=1}^{k-1} \frac{l}{k} u_l u_{k-l} , v )_{\Omega} - ( \nabla \cdot (\gamma - i k \omega \beta) \nabla u_k , v )_{\Omega} = 0
\end{align}
due to \eqref{eqn.111}-\eqref{eqn.112}, where we have used the Cauchy product formula \eqref{eqn.Cauchy} and term-by-term differentiation in time $t$ valid for absolutely convergent series. If $k$ is a non-positive integer, then we also get zero due to the orthogonality of the Fourier basis. Since the set $ \{ v e^{- i k \omega t} : v \in L^2(\Omega) \, , \, k \in \mathbb{Z} \} $   is dense in $L^2(\Omega \times (0,T))$, then $u \in C^2([0,T];H^2(\Omega))$ is a strong solution to the Westevelt equation where the equality in \eqref{eqn.001} is satisfied in the $L^2(\Omega \times (0,T))$-sense. Similarly, this solution $u$ possesses enough regularity to have Dirichlet trace in $C^2([0,T]; H^{3/2}(\partial \Omega))$ and Neumann trace in $C^2([0,T]; H^{1/2}(\partial \Omega))$. Therefore, due to the orthogonality of the Fourier basis, the boundary condition \eqref{eqn.002} is also satisfied strongly. 
\end{proof}

%%%%%%%%%%%%%%%%%%%%%%%%%%%%%%%%%%%%
\subsection{Uniqueness of the solution}

The uniqueness of solutions is a consequence of the results shown in this subsection. Following \cite{Kaltenbacher2011}, we first obtain a continuity inequality for the linearized problem, which is then used to verify a contraction property from where the uniqueness for the nonlinear problem follows. We propose to work with a linear version of the Westervelt equation that purposely contains time-independent coefficients and where all of the nonlinearity is replaced by a time-periodic source. For the remainder of this Section, $v_t$ stands for $\partial_t v$ and $v_{tt}$ stands for $\partial_t^2 v$ which eases the notation.

\begin{lemma} \label{lemma.uniqueness}
Let us assume that $f \in H^{0}((0,T); H^0(\Omega))$, and that $\beta, \gamma \in W^{1,\infty}(\Omega)$ and $\lambda, \eta \in W^{0,\infty}(\partial\Omega)$ are real-valued functions satisfying \eqref{eqn.lemma_unique_lower_bounds}.
Then, any time-periodic solution 
\begin{align}
v \in H^2((0,T);H^2(\Omega))
\label{eqn.lemma_unique_regularity}
\end{align}
to the linear system
\begin{equation}\label{eq:linear_system}
\begin{aligned}
    v_{tt} - \nabla \cdot( \gamma \nabla v) - \nabla \cdot ( \beta \nabla v_t)=f\qquad &\text{in $(0,T) \times \Omega$}\\
    ( \gamma + \beta \partial_{t} ) \partial_{\nu} v + \lambda v_t + \eta v= 0 \qquad &\text{on $(0,T) \times \partial \Omega$} \\
    v|_{t=0} = v|_{t=T} \qquad &\text{in $\Omega$} \\
    v_t|_{t=0} = v_t|_{t=T} \qquad &\text{in $\Omega$}
\end{aligned}
\end{equation}
satisfies the estimate 
\begin{align}
\|v\|_{H^{0}((0,T);H^1(\Omega))} + \|v_t\|_{H^{0}((0,T);H^{1}(\Omega))} + \|v_{tt}\|_{H^0((0,T);H^0(\Omega))} \leq C\|f\|_{H^0((0,T);H^0(\Omega))}.
\end{align}
\end{lemma}

\begin{proof}
We multiply the equation by $\overline{v}_t$, integrate over $\Omega$, and take the real part to obtain
\begin{align} \label{eq:lemma4_1}
\frac{1}{2} \frac{d}{dt} \left( \|v_t\|^2_{H^0(\Omega)} + \| \sqrt{\gamma} \nabla v \|^2_{H^0(\Omega)} + \| \sqrt{\eta} v \|^2_{H^0(\partial\Omega)} \right) + \| \sqrt{\beta} \nabla v_t\|^2_{H^0(\Omega)} + \|\sqrt{\lambda}v_t\|^2_{H^0(\partial\Omega)}=\Re (f,\overline{v}_t)_{H^0(\Omega)}.
\end{align}
Integrating with respect to time and using the periodicity of $v$ gives
\begin{align} \label{eq:lemma4_2}
\| \sqrt{\beta} \nabla v_t\|^2_{H^0((0,T);H^0(\Omega))} + \|\sqrt{\lambda}v_t\|^2_{H^0((0,T);H^0(\partial\Omega))}= \int^T_0 \Re (f,\overline{v}_t)_{H^0(\Omega)}dt ,
\end{align}
where the right-hand side is bounded by
\begin{align*}
\delta^{-1}\|f\|^2_{H^0((0,T);H^0(\Omega)))} + \delta\|v_t\|^2_{H^0((0,T);H^0(\Omega))},
\end{align*}
for any $\delta>0$. Applying the Poincar\'e-Friedrichs inequality (see \cite[\S 6.3]{Atkinson-Han-Book-2001}) on the left-hand side of \eqref{eq:lemma4_2} and taking $\delta>0$ sufficiently small allows us to absorb the second term with the left-hand side of the inequality and deduce the estimate
\begin{align}\label{eq:energy_ineq1}
\|v_t\|_{H^0((0,T);H^1(\Omega))}\leq C\|f\|_{H^0((0,T);H^0(\Omega))},
\end{align}
for a constant $C$ depending on $\Omega$,  $\beta_0$ and $\lambda_0$.

We multiply now the wave equation in \eqref{eq:linear_system} by $\overline{v}$, integrate over $(0,T)\times\Omega$, apply integration by parts in time and in space, and take the real part to obtain
\begin{align*}
-\|v_t\|^2_{H^0((0,T);H^0(\Omega))} + \| \sqrt{\gamma} \nabla v\|^2_{H^0((0,T);H^0(\Omega))} - \int^T_0 \Re ( (\gamma + \beta \partial_t) \partial_\nu v,\overline{v})_{H^0(\partial\Omega)}dt = \int^T_0 \Re (f,\overline{v})_{H^0(\Omega)}dt.
\end{align*}
From the boundary condition in \eqref{eq:linear_system} and time-periodicity, we get that
\begin{align*}
-\int^T_0\Re( ( \gamma + \beta \partial_t)\partial_\nu v,\overline{v})_{H^0(\partial\Omega)}dt = \int^T_0 \frac{1}{2} \frac{d}{dt} \|\sqrt{\lambda}v\|^2_{H^0(\partial\Omega)} + \|\sqrt{\eta}v\|^2_{H^0(\partial\Omega))} dt = \|\sqrt{\eta}v\|^2_{H^0((0,T);H^0(\partial\Omega))}.
\end{align*}
Therefore, the previous equality rewrites as
\begin{align*}
\| \sqrt{\gamma} \nabla v \|^2_{H^0((0,T);H^0(\Omega))} + \|\sqrt{\eta}v\|^2_{H^0((0,T);H^0(\partial\Omega))} =\int^T_0\Re(f,\overline{v})_{H^0(\Omega)}dt+\|v_t\|^2_{H^0((0,T);H^0(\Omega))}.
\end{align*}
Using inequality \eqref{eq:energy_ineq1} we see that
\begin{align*}
\gamma_0 \|\nabla v\|^2_{H^0((0,T);H^0(\Omega))} 
 + \eta_0 \|v\|^2_{H^0((0,T);H^0(\partial\Omega))} \leq 
 \left( \delta^{-1} + C^2 \right) \|f\|_{H^0((0,T);H^0(\Omega))}^2 + \delta \|v\|_{H^0((0,T);H^0(\Omega))}^2,
\end{align*}
with $C>0$ the constant from inequality \eqref{eq:energy_ineq1}. An application of Poincar\'e-Friedrichs inequality (see \cite[\S 6.3]{Atkinson-Han-Book-2001}) and choosing $\delta>0$ small enough gives the estimate
\begin{align*}
\|v\|_{H^0((0,T);H^1(\Omega))}\leq C\|f\|_{H^0((0,T);H^0(\Omega))},
\end{align*}
for a different constant $C>0$, depending on $\Omega, \beta_0, \gamma_0, \lambda_0, \eta_0$.

For the last part of the energy estimate, we multiply the linear equation by $\overline{v}_{tt}$, integrate over $\Omega$, apply integration by parts, and take real part to get
\begin{align*}
\|v_{tt}\|^2_{H^0(\Omega)} + \Re( \gamma \nabla v,\nabla \overline{v}_{tt})_{H^0(\Omega)} + \Re( \beta \nabla v_t, \nabla \overline{v}_{tt})_{H^0(\Omega)} + \Re( \lambda v_t + \eta v,\overline{v}_{tt})_{H^0(\partial\Omega)} = \Re(f,\overline{v}_{tt})_{H^0(\Omega)}.
\end{align*}
Using the identities
\begin{align*}
& \Re( \gamma \nabla v,\nabla\overline{v}_{tt})_{H^0(\Omega)} =  \frac{d}{dt} \| \sqrt{\gamma} \nabla v \|_{H^0(\Omega)} - \| \sqrt{\gamma} \nabla v_t\|^2_{H^0(\Omega)},\\
&\Re(\beta \nabla v_t , \nabla \overline{v}_{tt})_{H^0(\Omega)} = \frac{1}{2} \frac{d}{dt} \|\sqrt{\beta}\nabla v_t\|^2_{H^0(\Omega)},\\
&\Re(\lambda v_t,\overline{v}_{tt})_{H^0(\partial\Omega)} = \frac{1}{2} \frac{d}{dt} \|\sqrt{\lambda}v_t\|^2_{H^0(\partial\Omega)},
\end{align*}
and the periodicity of $v$, after integration in time over the interval $(0,T)$ one obtains
\begin{align*}
\|v_{tt}\|^2_{H^0((0,T);H^0(\Omega))}= \| \sqrt{\gamma} \nabla v_t\|^2_{H^0((0,T);H^0(\Omega))} + \int^T_0\left(\Re(f,\overline{v}_{tt})_{H^0(\Omega)}-(\eta v,\overline{v}_{tt})_{H^0(\partial\Omega)}\right)dt.
\end{align*}
We then use inequality \eqref{eq:energy_ineq1}, integration by parts, periodicity and the estimate
\begin{align*}
\int^T_0\Re(f,\overline{v}_{tt})_{H^0(\Omega)}dt\leq \delta^{-1}\|f\|^2_{H^0((0,T);H^0(\Omega)))} + \delta\|v_{tt}\|^2_{H^0((0,T);H^0(\partial\Omega))},
\end{align*}
which holds for any $\delta>0$, to get
\begin{align*}
\|v_{tt}\|^2_{H^0((0,T);H^0(\Omega)))} & \leq (C+\delta^{-1})\|f\|^2_{H^0((0,T);H^0(\Omega))} + \delta\|v_{tt}\|^2_{H^0((0,T);H^0(\Omega))} + \|\sqrt{\eta}v_{t}\|^2_{H^0((0,T);H^0(\partial\Omega))}\\
&\leq (C+\delta^{-1})\|f\|^2_{H^0((0,T);H^0(\Omega))} 
 + \delta\|v_{tt}\|^2_{H^0((0,T);H^0(\Omega))} + C'\|v_{t}\|^2_{H^0((0,T);H^1(\Omega))}\\
&\leq C''\|f\|^2_{H^0((0,T);H^0(\Omega))} 
 + \delta\|v_{tt}\|^2_{H^0((0,T);H^0(\Omega))}
\end{align*}
where we used a trace inequality and \eqref{eq:energy_ineq1}. We conclude the proof by taking $\delta>0$ small enough and absorbing the last term on the right.
\end{proof}

%%%%%%%%%%%%%%%%%%%%%%%%%
\begin{theorem}[Uniqueness] \label{thm.uniqueness}
Let $\Omega \subset \mathbb{R}^d$ for $d=2,3$ be a bounded connected domain with smooth boundary $\partial \Omega$. Assume that $\alpha \in L^{\infty}(\Omega)$,  $\beta, \gamma \in W^{1,\infty}(\Omega)$, and $\lambda, \eta \in W^{0,\infty}(\partial\Omega)$ are real-valued functions satisfying \eqref{eqn.lemma_unique_lower_bounds}. 
For any fixed $M>0$ there is $\epsilon>0$ for which, if $\| \alpha \|_{L^\infty(\Omega)}<\epsilon$, there exists at most one solution to \eqref{eqn.001}-\eqref{eqn.004} such that
\begin{align} \label{eqn.main_uniqueness}
\|u\|_{C^2([0,T];H^2(\Omega))}\leq M.
\end{align}
\end{theorem}

\begin{proof}
Let $u_1$ and $u_2$ be solutions to the nonlinear system \eqref{eqn.001}-\eqref{eqn.004} satisfying \eqref{eqn.main_uniqueness}. Let's set $v=u_1-u_2$. The function $v$ satisfies the equation
\begin{align*}
v_{tt}  - \nabla\cdot( \gamma \nabla v) - \nabla \cdot (\beta \nabla v_t) &= \alpha ( u_1^2 - u_2^2)_{tt} = \alpha \left( v (u_1 + u_2)  \right)_{tt} \\
&= \alpha \left(  v \, (u_1 + u_2)_{tt} + 2 v_t \, (u_1 + u_2)_{t} + v_{tt} \, (u_1 + u_2) \right).
\end{align*}
Thus, it can be regarded as a solution to the linear system \eqref{eq:linear_system} with source 
\begin{align*}
f = \alpha \left(  v \, (u_1 + u_2)_{tt} + 2 v_t \, (u_1 + u_2)_{t} + v_{tt} \, (u_1 + u_2) \right).
\end{align*}
From Lemma \ref{lemma.uniqueness}, it follows that
\begin{align*}
\|v\|_{H^0((0,T);H^1(\Omega))}+\|v_t\|_{H^0((0,T);H^1(\Omega))}+\|v_{tt}\|_{H^0((0,T);H^0(\Omega))}\leq C\|f\|_{H^0((0,T);H^0(\Omega))},
\end{align*}
where, after the use of the assumption \eqref{eqn.main_uniqueness} and the fact that the space $H^2(\Omega)$ is continuously embedded in $L^{\infty}(\Omega)$ in dimension $d=2,3$, we obtain
\begin{align*}
\|f\|_{H^0((0,T);H^0(\Omega))}\leq C M \|\alpha\|_{L^\infty(\Omega)} \left( \|v\|_{H^0((0,T);H^0(\Omega))} + \|v_t \|_{H^0((0,T);H^0(\Omega))} +  \|v_{tt}\|_{H^0((0,T);H^0(\Omega))} \right)
\end{align*}
for a constant $C>0$ independent of $M$. It is clear that for any choice of $M$, there is $\epsilon>0$ such that, for $\| \alpha \|_{L^\infty(\Omega)}<\epsilon$ then $v=0$.
\end{proof}

%%%%%%%%%%%%%%%%%%%%%%%%%%%%%%%%%%%
%%%%%%%%%  NEW SECTION  %%%%%%%%%%%
%%%%%%%%%%%%%%%%%%%%%%%%%%%%%%%%%%%
\section{Inverse problem using complex time-periodic solutions}
\label{Section.InverseProblem}

Due to the well-posedness of the forward problem \eqref{eqn.001}-\eqref{eqn.004} established in the previous sections, it is possible to define a Robin-to-Dirichlet operator that maps the boundary data $g \in H^{s}(\partial \Omega)$ for $s \geq 1/2$ to the Dirichlet trace of the solution $u$. This trace would belong to $C^2([0,T]; H^{3/2}(\partial \Omega))$. Note that this mapping depends nonlinearly on the coefficients of the Westervelt equation, namely, the (squared) wave speed $\gamma$, the diffusivity $\beta$, nonlinear coefficient $\alpha$, and the frequency $\omega$. Since the solution is time periodic, then $u$ can be uniquely decomposed into its Fourier components
\begin{align} \label{eqn.FourierAgain}
u(t,x) = \sum_{k=0}^{\infty} u_{k}(x) e^{- i k \omega t}
\end{align}
where $u_{k}$ is known as the $k$-th harmonic. In this section, we are interested only in the first-harmonic $u_{1}$ and second-harmonic $u_{2}$. The Dirichlet traces of these harmonics define the Robin-to-Dirichlet operator of interest, denoted by 
$\Lambda_{\alpha,\beta,\gamma}$, that maps
\begin{align} \label{eqn.MainNtD}
\Lambda_{\alpha,\beta,\gamma} : g \mapsto (u_{1}|_{\partial \Omega}, u_{2}|_{\partial \Omega}).
\end{align}
The objective of this section is to show that this map $\Lambda_{\alpha,\beta,\gamma}$ contains enough information to determine uniquely $\alpha$, $\beta$ and $\gamma$ under appropriate assumptions detailed below.

\subsection{First-harmonic measurements}
We begin by addressing the problem of determining the diffusivity $\beta$ and the (squared) wave speed $\gamma$ from measurements of first harmonic $u_1$ at the boundary $\partial \Omega$. Denote by $\Lambda^1_{\beta,\gamma}$ the Robin-to-Dirichlet (RtD) operator
\begin{align} \label{eqn.Lambda_mu}
\Lambda^1_{\beta,\gamma}:g\mapsto u_1|_{\partial\Omega}
\end{align}
with $u_1$ the unique solution to \eqref{eqn.111} assuming that the boundary coefficients $\lambda$ and $\eta$ are known and fixed. Let $\mu = \gamma-i\omega\beta$ and
denote by $\mu^{1/2}$ the principal square root of $\mu$, we apply the well-known Liouville transformation $U = \mu^{1/2} u_1$ and verify that $U$ solves the following elliptic equation
\begin{equation} \label{eqn.BVP_U}
\begin{aligned} 
\Delta U + \omega^2 p \, U &= 0, \quad \text{in $\Omega$} \\
(\gamma - i \omega \beta) \partial_{\nu} U - (i \omega \lambda - \eta + \mu^{1/2} \partial_{\nu} \mu^{1/2}) U &= G, \quad \text{on $\partial \Omega$}
\end{aligned}
\end{equation}
for the complex-valued potential
\begin{align} \label{eqn.potential_p}
p = \mu^{-1} - \omega^{-2} \mu^{-1/2} \Delta \mu^{1/2}. 
\end{align}
If $\beta, \gamma\in W^{2,\infty}(\Omega)$, then $p\in L^\infty(\Omega)$. This transformed equation defines an analogous Robin-to-Dirichlet (RtD) map associated with \eqref{eqn.BVP_U},  
\begin{align} \label{eqn.Lambda_p}
\Lambda_p : G \mapsto U|_{\partial\Omega}
\end{align}
where $G = \mu^{1/2} g$. Notice that given the Robin data $G$, the RtD map $\Lambda_{p}$, and the boundary condition in \eqref{eqn.BVP_U}, one can also obtain the Neumann data $\partial_{\nu} U$ as follows,
\begin{align*}
\partial_{\nu} U = \frac{G + (i \omega \lambda - \eta + \mu^{1/2} \partial_{\nu} \mu^{1/2} ) \Lambda_{p} G }{ \gamma - i \omega \beta }.
\end{align*}
Hence, from the RtD map $\Lambda_{p}$, one obtains the Cauchy data on $\partial \Omega$. The unique determination of a complex-valued and bounded potential $p$ from all pairs of Cauchy data has previously been addressed in the literature. At the core of its proof lies the construction of appropriate complex geometric optics (CGO) solutions. As noted in \cite{Lakshtanov2016}, for the 2-dimensional case, uniqueness is obtained from the work in \cite{Nachman1996,Bukhgeim2008} which combines a specific CGO construction along with the stationary phase method, while for dimensions 3 (and higher), as remarked in \cite{Borcea2002EIT,Pombo2022}, the result follows from a standard CGO construction and Fourier inversion from \cite{Sylvester1987,Nachman1988,Nachman1988a,Nagayasu2013}. Since knowledge of the map $\Lambda_{p}$ implies knowledge of the information in the standard DtN or NtD maps, we conclude that the aforementioned results give the injectivity of the map $p\mapsto \Lambda_p$ with $\Lambda_p$ defined in \eqref{eqn.Lambda_p}.

Now we proceed to show the injectivity of the map $\mu^{1/2} \mapsto p$ defined by \eqref{eqn.potential_p} to conclude the unique determination of $\mu^{1/2}$, and consequently of $\beta$ and $\gamma$, from knowledge of $p$. Let us define $\mu = \gamma - i \omega \beta$ and $\tilde{\mu} = \tilde{\gamma} - i \omega \tilde{\beta}$, and assume that $\gamma, \tilde{\gamma}, \beta, \tilde{\beta} \in W^{2,\infty}(\Omega)$ satisfying \eqref{eqn.lemma_unique_lower_bounds}, are such that $\mu = \tilde{\mu}$ on $\partial \Omega$ and 
\begin{align*}
p = \mu^{-1} - \omega^{-2} \mu^{-1/2} \Delta \mu^{1/2} = \tilde{\mu}^{-1} - \omega^{-2} \tilde{\mu}^{-1/2} \Delta \tilde{\mu}^{1/2}.
\end{align*}
Then
\begin{align*}
(\mu^{1/2} - \tilde{\mu}^{1/2}) p = \mu^{-1/2} - \tilde{\mu}^{-1/2} - \omega^{-2} \Delta (\mu^{1/2} - \tilde{\mu}^{1/2})
\end{align*}
or equivalently
\begin{equation} \label{eqn.nu}
\begin{aligned}
\Delta \nu + \omega^{2} ( p + (\mu \tilde{\mu})^{-1/2} ) \nu &= 0 \qquad \text{in $\Omega$} \\
\nu&= 0 \qquad \text{on $\partial \Omega$}
\end{aligned}
\end{equation}
where $\nu = \mu^{1/2} - \tilde{\mu}^{1/2}$. We work in the high-frequency regime as $\omega \to \infty$ where
\begin{align*}
\mu^{-1} &= \frac{1}{\gamma - i \omega \beta} = \frac{i}{\omega \beta} \left( 1 + \mathcal{O}(1/\omega) \right) \\
\mu^{-1/2} &= \frac{1+i}{\sqrt{2}}  \frac{1}{(\omega \beta)^{1/2}} \left( 1 + \mathcal{O}(1/\omega) \right)
\\
\mu^{1/2} &= \frac{1-i}{\sqrt{2}} (\omega \beta)^{1/2} \left( 1 + \mathcal{O}(1/\omega) \right) \\
p &= \frac{i}{\omega \beta} - \frac{1}{\omega^2} \beta^{-1/2} \Delta \beta^{1/2} + \mathcal{O}(1/\omega^3)
\end{align*}
where analogous expressions are valid for $\tilde{\mu}$. Therefore,
\begin{align*}
p + (\mu \tilde{\mu})^{-1/2} &= \frac{i}{\omega \beta} + \frac{(1+i)^2}{2 \omega (\beta \tilde{\beta})^{1/2}}  \left( 1 + \mathcal{O}(1/\omega) \right) + \mathcal{O}(1/\omega^3) \\
&= \frac{i}{\omega \beta^{1/2}} \left( \beta^{-1/2} + \tilde{\beta}^{-1/2} \right) + \mathcal{O}(1/\omega^2).
\end{align*}

As a result, for sufficiently large frequency $\omega=\omega(\gamma,\beta,\gamma_{0},\beta_{0})$, the term $p + (\mu \tilde{\mu})^{-1/2}$ has an imaginary part that is uniformly bounded away from zero. Consequently, the only solution to the BVP \eqref{eqn.nu} is $\nu \equiv 0$. The above arguments prove the following result.

\begin{lemma} \label{Lemma.Harmonic_1}
Let $\Omega \subset \mathbb{R}^d$ for $d=2,3$, be a bounded connected domain with a smooth boundary $\partial \Omega$. Assume that $\beta, \tilde{\beta}, \gamma, \tilde{\gamma} \in  W^{2,\infty}(\Omega)$, and $\lambda, \eta \in L^{\infty}(\partial\Omega)$ are all real-valued functions satisfying \eqref{eqn.lemma_unique_lower_bounds}. Moreover, we assume $\|\alpha\|_{L^\infty(\Omega)}, \|\tilde{\alpha}\|_{L^\infty(\Omega)} < \epsilon$ for a sufficiently small $\epsilon>0$ guaranteeing the existence and uniqueness of solutions to \eqref{eqn.001}-\eqref{eqn.004} according to Theorems \ref{thm.wellpose_forward} and \ref{thm.uniqueness}. Define $\mu = \gamma - i \omega \beta$ and $\tilde{\mu} = \tilde{\gamma} - i \omega \tilde{\beta}$ and assume that $\mu = \tilde{\mu}$ and $\partial_{\nu} \mu = \partial_{\nu} \tilde{\mu}$ on $\partial \Omega$. Denote by $\Lambda^{1}_{\beta,\gamma}$ and $\Lambda^{1}_{\tilde{\beta},\tilde{\gamma}}$ the RtD maps for the first-harmonic data defined by \eqref{eqn.Lambda_mu}.
Then there exists a frequency $\omega_{0}=\omega_{0}(\gamma,\beta,\gamma_{0},\beta_{0})$ such that if $\omega>\omega_{0}$ and $\Lambda^{1}_{\beta,\gamma} = \Lambda^{1}_{\tilde{\beta},\tilde{\gamma}}$, then $\beta=\tilde{\beta}$ and $\gamma=\tilde{\gamma}$.
\end{lemma}

This lemma states that under some technical conditions of regularity for the coefficients of the Westervelt equation and sufficiently large oscillatory frequency $\omega$, the Robin-to-Dirichlet map $\Lambda_{\beta,\gamma}^{1}$, defined in \eqref{eqn.Lambda_mu}, associated with measuring the first harmonic of the solution at the boundary $\partial \Omega$,  
determines the coefficient of diffusivity $\beta$ and the (squared) wave speed $\gamma$ uniquely in $\Omega$.

%%%%%%%%%%%%%%%%%%%%%%%%%%%%%%%%%%%%%%%%%%%%%%%%%%
\subsection{Second-harmonic measurements}
Let's assume that $\beta$ and $\gamma$ are fixed and known as a result of the inversion of first-harmonic data, and switch to the inverse problem of determining the coefficient of nonlinearity $\alpha$. Therefore, in this subsection we omit the dependence of the second-harmonic component of the acoustic field on $\beta$ and $\gamma$.

The solution to the inverse problem is based on the following observation involving second-harmonic measurements. Recalling the governing system \eqref{eqn.112} for $u_2$, it simplifies to
\begin{equation} \label{eqn.601}
\begin{aligned}
(2 \omega)^2 u_{2}  + \nabla \cdot (\gamma - 2 i \omega \beta) \nabla u_2 &=  (2 \omega)^2 \,  \alpha \, u_{1}^2, \qquad  \text{in $\Omega$},  \\
( \gamma - 2 i \omega \beta ) \partial_{\nu} u_2 - (2i \omega \lambda - \eta) u_2 &= 0, \qquad  \text{on $\partial \Omega$}. 
\end{aligned}
\end{equation}
We define an auxiliary linear problem as follows,
\begin{equation} \label{eqn.605}
\begin{aligned}
(2 \omega)^2 v + \nabla \cdot (\gamma - 2 i \omega \beta) \nabla v &=  0, \qquad  \text{in $\Omega$},  \\
( \gamma - 2 i \omega \beta ) \partial_{\nu} v - (2 i \omega \lambda - \eta) v &= -h, \qquad  \text{on $\partial \Omega$,} 
\end{aligned}
\end{equation}
for some admissible (in terms of regularity) but otherwise arbitrary boundary source $h$. Now multiply \eqref{eqn.601} by $v$, and \eqref{eqn.605} by $u_2$, integrate in $\Omega$ both equations, apply  integration by parts, and plug one into the other to obtain that
\begin{align} \label{eqn.609}
(2\omega)^2 \int_{\Omega} \alpha u_{1}^2  v \, dV = \int_{\partial \Omega} h \, u_2  \, dS = \int_{\partial \Omega} h \, \Lambda^2_{\alpha} g  \, dS
\end{align}
where the Robin-to-Dirichlet map $\Lambda^2_{\alpha}$ is defined as
\begin{align} \label{eqn.610}
\Lambda^2_{\alpha} : g \mapsto u_2|_{\partial \Omega}
\end{align}
for the boundary source $g$ that excites $u_1$ in \eqref{eqn.111}, which in turn excites $u_2$ in \eqref{eqn.601} as an interior source. It only remains to show that the set $\{ u_{1}^2 v \}$ constructed from all admissible boundary sources $g$ and $h$, and frequencies $\omega \geq \omega_{o}$ is dense in the space $H^0(\Omega)$ to conclude that knowledge of the map $\Lambda^2_{\alpha}$ is sufficient to determine the coefficient of nonlinearity $\alpha$. In fact, given another coefficient of nonlinearity $\tilde{\alpha}$, then the analogue of the identity \eqref{eqn.609} can be derived, and the difference between the identities reads
\begin{align} \label{eqn.611}
(2\omega)^2 \int_{\Omega}  ( \alpha - \tilde{\alpha} ) \, u_{1}^2 v \, dV = \int_{\partial \Omega} h \, (\Lambda^2_{\alpha} - \Lambda^2_{\tilde{\alpha}}) g \, dS
\end{align}
valid for all solutions $u_1$ of \eqref{eqn.111} and all solutions $v$ of \eqref{eqn.605}, for arbitrary frequency $\omega \geq \omega_{o}$. Therefore, provided that the set $\{ u_{1}^2 v \}$ is dense in $H^0(\Omega)$, then, the second-harmonic RtD map $\Lambda_{\alpha}^2$ would determine the coefficient of nonlinearity $\alpha$. 

To demonstrate the denseness of the products $\{ u_{1}^2 v \}$ we perform the same Liouville transformation $U = \mu^{1/2} u_1$ where $\mu = \gamma - i \omega \beta$, thus
\begin{align} \label{eqn.621}
\Delta U + \omega^2 p \, U = 0, \quad \text{in $\Omega$, where} \quad
p = \mu^{-1} - \omega^{-2} \mu^{-1/2} \Delta \mu^{1/2}
\end{align}
and similarly we set $V = \mu_2^{1/2} v$ for $\mu_2 = \gamma - 2 i \omega \beta$, and verify that $V$ solves
\begin{align} \label{eqn.625}
\Delta V + (2 \omega)^2 p_2 \, V = 0,
\quad \text{in $\Omega$, where} \quad
p_2 = \mu_{2}^{-1} - (2 \omega)^{-2} \mu_2^{-1/2} \Delta \mu_2^{1/2}
\end{align}
and with $\mu_2^{1/2}$ the principal square root of $\mu_2$. The denseness of $\{ U^2V \}$ (and consequently of $\{ u_{1}^2 v \}$) follows from the construction of CGO solutions to \eqref{eqn.621} and \eqref{eqn.625}. Similarly as in the previous section, we follow \cite{Nachman1996,Bukhgeim2008} for dimension 2 and \cite{Sylvester1987,Nachman1988,Nachman1988a,Nagayasu2013,Pombo2022} for dimension 3. In particular, for dimension $d=3$ as in \cite[Prop. 2.3]{Nagayasu2013} and references therein, for a given $k \in \mathbb{R}^3$, 
one can find $\zeta_1, \zeta_2 \in \mathbb{C}^3$ such that $\zeta_1 \cdot \zeta_1 = \zeta_2 \cdot \zeta_2 = 0$, and $\zeta_1 + \zeta_2 = i k$, while the magnitudes $|\zeta_1|$ and $|\zeta_2|$ can be made arbitrarily large. The corresponding CGO functions of the form $U(x) = e^{\zeta_1 \cdot x/2} ( 1 + \psi_{U}(x) )$ and $V = e^{ \zeta_2 \cdot x} (1 + \psi_{V}(x))$ solve \eqref{eqn.621} and \eqref{eqn.625}, respectively, such that
\begin{align*}
\| \psi_{U} \|_{H^s(\Omega)} \leq \frac{C_1 \omega^2 }{|\zeta_1|} \| p \|_{H^{s}(\Omega)} \quad \text{and} \quad \| \psi_{V} \|_{H^s(\Omega)} \leq \frac{C_2 \omega^2 }{|\zeta_2|} \| p_2 \|_{H^{s}(\Omega)}
\end{align*}
valid for all sufficiently large $|\zeta_1|, |\zeta_2|$, for some constants $C_1, C_2 > 0$ and for some integer $s \geq 2$. As result, as we take $|\zeta_1|, |\zeta_2| \to \infty$, then $U^2 V \to e^{i k \cdot x}$ uniformly in $\Omega$. Since $k \in \mathbb{R}^3$ is arbitrary, then this argument shows the denseness of the products $\{ U^2 V \}$ via the Fourier transform.

The previous analysis leads to the next uniqueness result.

\begin{lemma} \label{Lemma.Harmonic_2}
Let $\Omega \subset \mathbb{R}^d$ for $d=2,3$, be a bounded connected domain with a smooth boundary $\partial \Omega$. Assume that $\beta, \gamma \in W^{2,\infty}(\Omega)$, and $\lambda, \eta \in L^{\infty}(\partial\Omega)$ are all real-valued functions satisfying \eqref{eqn.lemma_unique_lower_bounds}. Moreover, we assume $\|\alpha\|_{L^\infty(\Omega)}, \|\tilde{\alpha}\|_{L^\infty(\Omega)} < \epsilon$ for a sufficiently small $\epsilon>0$ guaranteeing the existence and uniqueness of solutions to \eqref{eqn.001}-\eqref{eqn.004} according to Theorems \ref{thm.wellpose_forward} and \ref{thm.uniqueness}. Denote by $\Lambda_{\alpha}^{2}$ and $\Lambda_{\tilde{\alpha}}^{2}$ the RtD maps for the second-harmonic data defined by \eqref{eqn.610}. There exists a frequency $\omega_0$ such that if $\omega > \omega_{0}$ and $\Lambda_{\alpha}^{2} = \Lambda_{\tilde{\alpha}}^{2}$, then $\alpha = \tilde{\alpha}$.
\end{lemma}

%%%%%%%%%%%%%%%%%%%%%%%%%%%%%%%%%%%%%%%%%%%%%%%%%%
\subsection{Simultaneous determination of wave speed, diffusivity and nonlinearity}

Now we proceed to combine the Lemmas \ref{Lemma.Harmonic_1}-\ref{Lemma.Harmonic_2} from the previous two subsections and the definition of the RtD map \eqref{eqn.MainNtD} for boundary data belonging to the first and second harmonics which can be written as
\begin{align} \label{eqn.NtD_final}
\Lambda_{\alpha,\beta,\gamma} = ( \Lambda_{\beta,\gamma}^{1} , \Lambda_{\alpha}^{2})
\end{align}
where $\Lambda_{\beta, \gamma}^{1}$ is defined in \eqref{eqn.Lambda_mu} and $\Lambda_{\alpha}^{2}$ is defined in 
\eqref{eqn.610}. These previous results straightforwardly lead to the main theorem of this paper.

\begin{theorem}[Uniqueness for the inverse problem] \label{thm.uniqueness_inverse}
Let $\Omega \subset \mathbb{R}^d$ for $d=2,3$, be a bounded connected domain with a smooth boundary $\partial \Omega$. Assume that $\beta, \tilde{\beta}, \gamma, \tilde{\gamma} \in  W^{2,\infty}(\Omega)$, and $\lambda, \eta \in L^{\infty}(\partial\Omega)$ are all real-valued functions satisfying \eqref{eqn.lemma_unique_lower_bounds}. Moreover, we assume $\|\alpha\|_{L^\infty(\Omega)}, \|\tilde{\alpha}\|_{L^\infty(\Omega)} < \epsilon$ for a sufficiently small $\epsilon>0$ guaranteeing the existence and uniqueness of solutions to \eqref{eqn.001}-\eqref{eqn.004} according to Theorems \ref{thm.wellpose_forward} and \ref{thm.uniqueness}. Define $\mu = \gamma - i \omega \beta$ and $\tilde{\mu} = \tilde{\gamma} - i \omega \tilde{\beta}$ and assume that $\mu = \tilde{\mu}$ and $\partial_{\nu} \mu = \partial_{\nu} \tilde{\mu}$ on $\partial \Omega$. 
Denote by $\Lambda_{\alpha,\beta,\gamma}$ and $\Lambda_{\tilde{\alpha},\tilde{\beta},\tilde{\gamma}}$ the RtD maps defined by \eqref{eqn.MainNtD}. There exists a frequency $\omega_0$ such that if $\omega > \omega_{0}$ and $\Lambda_{\alpha,\beta,\gamma} = \Lambda_{\tilde{\alpha},\tilde{\beta},\tilde{\gamma}}$, then $\alpha = \tilde{\alpha}$, $\beta = \tilde{\beta}$, and $\gamma=\tilde{\gamma}$ in $\Omega$. 
\end{theorem}

%%%%%%%%%%%%%%%%%%%%%%%%%%%%%%%%%%%
%%%%%%%%%  NEW SECTION  %%%%%%%%%%%
%%%%%%%%%%%%%%%%%%%%%%%%%%%%%%%%%%%
\section{Discussion}
\label{Section.Discussion}

We first highlight the fact that the results of this paper complement previous works such as \cite{Acosta2022,Kaltenbacher2021b,Kaltenbacher2021a,Kaltenbacher2022a,Kaltenbacher2023,Kaltenbacher2022,Eptaminitakis2024,Kaltenbacher2024b} where the inverse problem for the Westervelt equation is considered under various assumptions. Here we focus on determining all three coefficients of the differential equation (wave speed, diffusivity, nonlinearity) using the whole (complex-valued) RtD map for the measurements corresponding to the first- and second-harmonics. We have only addressed uniqueness. The stability of this inverse problem remains to be analyzed. Our particular choice of boundary probing source \eqref{eqn.002} of the form $g e^{- i \omega t}$ and imposing periodicity conditions in time reduces the problem to a triangular system of elliptic equations for each harmonic. Hence, the mathematical machinery for elliptic inverse problems can be invoked. However, such an approach sacrifices the ability to establish stability or to devise accurate reconstruction algorithms. Future work could take better advantage of the hyperbolicity of the time-dependent Westervelt equation or use data from multiple, increasing frequencies $\omega$ to ensure stability and good resolution as done for instance in \cite{Gang2010,Borges2017} for inverse scattering problems. The authors are currently investigating this point numerically. As soon as meaningful results are obtained, they will be reported in an appropriate venue.

We acknowledge that we purposely employed complex-valued solutions to the Westervelt equation. This made it possible for the system of harmonic components to be triangular, ie., lower harmonics are independent of higher harmonics. As a result, the first-harmonic component does not depend on the coefficient of nonlinearity $\alpha$. This property simplifies the analysis considerably and conforms to the experimental motivation to generate images based on the nonlinear behavior \cite{Anvari2015,Hedrick2005,Thomas1998}. However, complex-valued solutions do not have a direct physical interpretation since taking their real part does not in general render a solution to the Westervelt equation \cite{Kaltenbacher2021}. Therefore, the simultaneous determination of the wave speed, diffusivity and nonlinearity for this Westervelt model from boundary traces of real-valued solutions remains an open problem.

%%%%%%%%%%%%%%%%%%%%%%%%%%%%%%%%%%%
%%%%%%%%%  NEW SECTION  %%%%%%%%%%%
%%%%%%%%%%%%%%%%%%%%%%%%%%%%%%%%%%%
\bigskip

\section*{Acknowledgment}
The work of B. Palacios was partially supported by Agencia Nacional de Investigaci\'on y Desarrollo (ANID), Grant FONDECYT Iniciaci\'on N$^\circ$11220772. The work of S. Acosta was partially supported by NIH award 1R15EB035359-01A1.
S. Acosta would like to thank the support and research-oriented environment provided by Texas Children's Hospital.

\bigskip
%\clearpage

\bibliographystyle{plain}
\bibliography{library}

%\bibliographystyle{alpha}
%\bibliography{sample}

\end{document}